\documentclass{dcds}
\usepackage{relsize}

\usepackage{amsmath,amssymb}
\usepackage[utf8]{inputenc}
\usepackage[english]{babel}
\usepackage{csquotes}
  \usepackage{paralist}
 \usepackage{psfrag, subfigure}
  \usepackage{graphics}
  \usepackage{epsfig}
\usepackage{graphicx}
\usepackage{epstopdf}
 \usepackage[colorlinks=true]{hyperref}
\hypersetup{urlcolor=blue, citecolor=red}

\def\currentvolume{}
 \def\currentissue{}
  \def\currentyear{}
   \def\currentmonth{}
    \def\ppages{X--XX}
     \def\DOI{}

\newtheorem{thm}{Theorem}[section]
\newtheorem{clly}{Corollary}[section]

\newtheorem{lema}{Lemma}[section]
\newtheorem{prop}{Proposition}[section]

\theoremstyle{definition}
\newtheorem{defi}{Definition}[section]
\newtheorem{rk}{Remark}[section]

\newcommand{\N}{\mbox{$\mathbb{N}$}}
\newcommand{\T}{\mbox{$\mathbb{T}$}}
\newcommand{\Z}{\mbox{$Z\!\!\!Z$}}

\newcommand{\R}{\mbox{$I\!\!R$}}

\numberwithin{equation}{section}

\makeatother

\begin{document}

\title[ Robust transitivity of singular endomorphisms ]{ A persistently singular map of $\T^n$ that is $C^1$ robustly transitive.}

\author[Juan C. Morelli]{}


\email{jmorelli@fing.edu.uy }

\maketitle

\centerline{\scshape  Juan Carlos Morelli$^*$}
\medskip
{\footnotesize
 \centerline{Universidad de La Rep\'ublica. Facultad de Ingenieria. IMERL.}
   \centerline{ Julio Herrera y Reissig 565. C.P. 11300.}
   \centerline{ Montevideo, Uruguay.}}

\bigskip

 \centerline{(Communicated by )}

\begin{abstract}

 We exhibit a $C^1$ robustly transitive endomorphism displaying critical points on the $n$-dimensional torus.

\end{abstract}

\section{Introduction}

Whenever we think about dynamical systems' properties almost inevitably come to mind the concepts of {\textit{stability}} and \textit{robustness}. Loosely speaking, we can say that stability implies same dynamics for maps sufficiently close to each other, and robustness implies the same behavior relative to a specifical property for maps sufficiently close to each other.\\

   This work in particular is focused in the study of robust transitivity, meaning by \textit{transitive} the existence of a forward dense orbit of a point. This may seem at first sight as an unexciting topic since a fair amount of results concerning robust transitivity are known. Nonetheless, the aimed class of maps, the singular endomorphisms about which little to nothing is known; as well as taking on the high dimensional context are undoubtedly a fresh approach to the subject.\\

   To set ideas in order we list up the most relevant known results about the topic.\\
   We begin summing up the most studied case: robust transitivity of diffeomorphisms. The image provided by known results is fairly complete. In the setting of surfaces, it is shown in \cite{m2} that robust transitivity implies the diffeomorphism to be Anosov and that the only surface that supports them is $\T^2$. Later on, in the arbitrary dimensional setting it is proven in \cite{bdp} that robust transitivity implies a dominated splitting on the tangent spaces (i.e. weak hyperbolicity). \\ Going further, next comes robust transitivity of regular endomorphisms (not globally but locally invertible). The image we have about these is somewhat less complete; yet we know that volume expanding is a necessary but not sufficient condition for $C^1$ robust transitivity according to \cite{lp}, where they also give a sufficient condition for the case of manifold $\T^n$.\\ Carrying on, at last there is the least studied case, robust transitivity of singular maps (non empty critical set). Until 2013 nothing had ever been written on the topic. It was on that year when \cite{br} showed the first example of a $C^1$ transitive singular map of $\T^2$. The second example was given only in 2016 by \cite{ilp}, they exhibit a $C^1$ robustly transitive map of $\T^2$ with a persistent critical set. Nothing more than these two examples was known until that time. Even so, there have been recent further advances on robust transitivity of singular surface endomorphisms: in 2018 \cite{ip} presented an example on $\T^2$ whose robust transitivity depends on the class of differentiabilty, and in 2019 \cite{lr1} and \cite{lr2} set the \textit{state of the art} proving that partial hyperbolicity is a necesary condition, that the only surfaces that support them are $\T^2$ and the Klein bottle, and that they belong to the homotopy class of a linear map with an eigenvalue of modulus larger than one. Finally, about singular endomorphisms in high dimensions, the only known result was given by \cite{mo} where he extended the result appearing in \cite{ip} to $\T^n$. \\

   In the spirit of generalizing known results in low to higher dimensions, the survey contained in our paper shows that the example exhibited in \cite{ilp} can also be extended to $\T^n$, resulting in the first known example of a persistently singular endomorphism\footnote{By persistently singular endomorphism we mean a map $f$ satisfying that there exists a $C^1$ neighborhood $\mathcal{U}_f$ of $f$ such that every map $g$ belonging to $\mathcal{U}_f$ displays critical points.} that is robustly transitive in the $C^1$ topology and supported on a manifold of dimension larger than $2$.\\
    The main result can be stated as:

   \begin{thm}\label{main}
     Given $n \geq 2$, there exists a persistently singular endomorphism supported on $\T^n$ that is $C^1$ robustly transitive.
   \end{thm}

    \subsection{Sketch of the Construction.}     Start from an endomorphism induced by a diagonal expanding matrix with integer coefficients, with all but one directions strongly unstable and one central direction. Perturb the map to add a blending region that mixes everything getting the transitivity, and then introduce artificially the critical points preserving the transitivity property. All this construction is done in a robust way.\\

    The author wants to remind the readers that the contents to follow are an adaptation of the surfaces' construction exhibited by \cite[Section 2.2]{ilp} to arbitrary dimensions. The proofs to some of the claims in our Lemmas and Theorems are, consequently, also inspired by \cite{ilp}. Moreover, many of them can be adapted in a straightforward manner cleverly enough, but in the sake of a self contained article all proofs will be explicitly provided here. Finally, if the readers wish to get a hollower approach to our construction by overviewing the low dimensional context first, they are gently invited to get in touch with the cited article.

\section{Preliminaries}

Some basic definitions are recalled at the beginning. If the readers wish to get more insight on geometrical or dynamical background they can refer themselves to \cite{gg} or \cite{kh}.\\

Let $M$ be a differentiable manifold of dimension $m$ and $f:M \rightarrow M$ a differentiable endomorphism. The \textbf{orbit} of $x \in M$ is $\mathcal{O}(x)=\{ f^n(x) , n \in \N \}$. The map $f$ is \textbf{transitive} if there exists a point $x \in M$ such that $ \overline{{\mathcal{O}(x)}}=M$ and $f$ is $C^k$-\textbf{robustly transitive} if there exists a neighborhood $\mathcal{U}_f$ of $f$ in the $C^k$ topology such that $g$ is transitive for all $ g$ belonging to $\mathcal{U}_f$.

The proposition ahead is well known and of most practical use.

\begin{prop}\label{equi}
  If $f$ is continuous then are equivalent:
  \begin{enumerate}
    \item  $f$ is transitive.
    \item  For all $U, V$ open sets in $M$, exists $n \in \N$ such that $ f^n(U) \cap V \neq \emptyset$.
    \item  There exists a residual set $R$ (countable intersection of open and dense sets) such that for all points $  x \in R: \overline{\mathcal{O}(x)}=M$.
  \end{enumerate}
\end{prop}

\subsection{Normally Hyperbolic (sub)manifolds}

We continue defining normally hyperbolic submanifolds in the sense of \cite{bb}. These kind of submanifolds for a given map are persistently invariant under perturbation, this allows defining dynamical systems within them. This will be the main usage we will make of them ahead in the paper. Their formal definition is as follows.

Let $f:M \to M$ be a $C^1$ diffeomorphism, $N \subset M$ a $C^1$ closed submanifold such that $f(N)=N$ (we say that $N$ is \textit{invariant}).
\begin{defi}
  We say that $f$ is \textit{Normally Hyperbolic at $N$} if there exists a splitting of the the tangent bundle of $M$ over $N $ into three $Df$-invariant subbundles such that $TM_{|N} = E^s \oplus E^u \oplus TN$ and a constant $0 < \lambda < 1$ such that for all $x \in \mathbb{N}$ the following hold:
  \begin{itemize}
    \item $||D_xf_{|E^s_x}|| < \lambda$, $||(D_xf)^{-1}_{|E^u_x}|| < \lambda$,
    \item $||D_xf_{|E^s_x}||.||(D_{f(x)}f)^{-1}_{|T_{f(x)}N}|| < \lambda$,
    \item $||(D_xf)^{-1}_{|E^u_x}||.||(D_{f^{-1}(x)}f)_{|T_{f^{-1}(x)}N}|| < \lambda$.
  \end{itemize}
 \end{defi}

 The first condition implies that the behavior of the differential map $Df$ is hyperbolic while the other two describe the domination property relative to stable $E^s$ and unstable $E^u$ subspaces. Our interest in these submaniolds comes from \cite[Theorem 2.1]{bb} which states that:
  \begin{thm}\label{NHSP}
     Given $M,N$ and $f$ as in the definition above, there exists $\mathcal{U}_f$ a $C^1$ neighborhood of $f$ such that all $g \in  \mathcal{U}_f$ admit a $C^1$ invariant submanifold $N_g$ which is unique such that $g$ is normally hyperbolic at $N_g$. Moreover, $N$ and $N_g$ are diffeomorphic and there exists an embedding from $N$ to $N_g$ which is $C^1$ close to the canonical inclusion $i:N \to M$.
  \end{thm}

 \subsection{Blenders.}

 A brief overview of the concept of a \textit{blender} is given now. In most situations it is easy to think of blenders as higher dimensional horseshoes, or as sets exhibiting the dynamics of a Smale's horseshoe. Blenders force the robust intersection of topologically 'thin' sets, giving rise to rich dynamics.\\ According to \cite{bcdw}, \begin{displayquote}{\textit{''A blender is a compact hyperbolic set whose unstable set has dimension strictly less than one would predict by looking at its intersection with families of submanifolds''.}}\end{displayquote}

 They also provide with a prototipical example of a blender: Let $R$ be a rectangle with two rectangles $R_1$ and $R_2$ lying inside, horizontally, and such that their projections onto the base of $R$ overlap (Figure \ref{protoblender}). Consider now a diffeomorphism $f$ such that $f(R_1)=f(R_2)=R$. Then, $\Omega= \bigcap_{n \in \N}f^{-n}(R)$ gives rise to a blender (Cantor) set for $f$. Observe that $f$ admits a fixed point inside each of $R_1$ and $R_2$, and that all vertical segments between the projection of these points intersect $\Omega$ (this is due to the overlapping of the projections of $R_1$ and $R_2$ which holds at every preiteration). Observe as well that this construction is robust in two senses: on the one hand, $f$ can be slightly perturbed with persistance of the property. And on the other, the vertical segment can also be slightly perturbed and still intersect $\Omega$.

 \begin{figure}[ht]
\begin{center}
\includegraphics[scale=0.45]{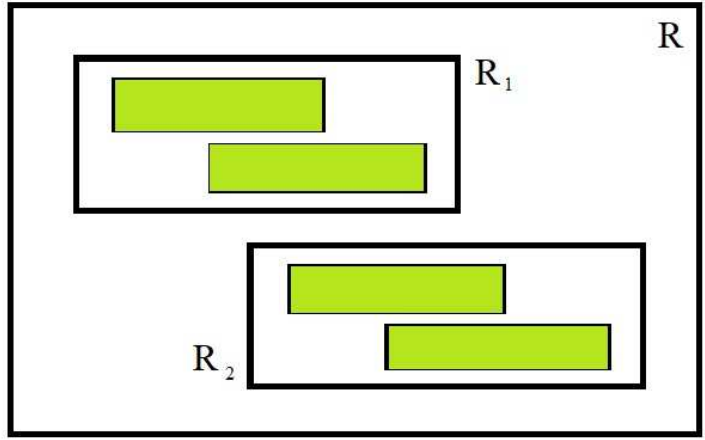}
\caption{A protoblender over $R$. Darker is $f^{-1}(R)$.}\label{protoblender}
\end{center}
\end{figure}

Notice that $\Omega$ is a fractal object with topological dimension zero. Nonetheless, every close-to-vertical line in between the fixed points of $f$ inside $R_1 \cup R_2$ intersects $\Omega$; hence, one would expect $\Omega$ to be at least of topological dimension one. This is the characteristical trait of blender sets.\\
To finish with the preliminaries regarding blenders, their importance lies in the fact that they are a magnificent tool for producing rich dynamics, particularly robustly transitive dynamics. For more insight on blenders and its applications the reader may go to \cite{bd}.

\subsection{Iterated Function Systems.}

Let $\mathcal{F},\mathcal{G}$ be two families of diffeomorphisms of $M$. Denote by $\mathcal{F} \circ \mathcal{G}:= \{f \circ g  / \quad f \in \mathcal{F}, g \in \mathcal{G}\}$; and for $k \in \N$ denote $\mathcal{F}^0=\{ Id_M\}$ and $\mathcal{F}^{k+1}=\mathcal{F}^{k} \circ \mathcal{F}$. Then, the set $\bigcup_{k=0}^{\infty}\mathcal{F}^k$ has a semigroup structure that is denoted by $\langle \mathcal{F}\rangle^+$ and said to be generated by $\mathcal{F}$. The action of the semigroup $\langle \mathcal{F}\rangle^+$ on $M$ is called the \textbf{iterated function system} associated with $\mathcal{F}$. We denote it by IFS$(\mathcal{F})$. For $x \in M$, the \textbf{orbit} of $x$ by the action of the semigroup $\langle \mathcal{F}\rangle^+$ is $ \langle \mathcal{F}\rangle^+(x)=\{f(x), f \in \langle \mathcal{F}\rangle^+ \}$. A sequence $\{ x_n, \quad n \in \N \}$ is a branch of an orbit of IFS$(\mathcal{F})$ if for every $n \in \N$ there exists $f_n \in \langle \mathcal{F}\rangle^+$ such that $f_n(x_n)=x_{n+1}$.
\begin{defi}
  An IFS$(\mathcal{F})$ is \textbf{minimal} if for every $x \in M$ the orbit $ \langle \mathcal{F}\rangle^+(x)$ has a branch that is dense on $M$. \\
  An IFS$(\mathcal{F})$ is \textbf{$C^r$ robustly minimal} if for every family $\mathcal{\hat{F}}$ of $C^r$ perturbations of $\mathcal{F}$ and every $x \in M$ the orbit $\langle \mathcal{\hat{F}}\rangle^+(x)$ has a branch that is dense on $M$.
\end{defi}

As is shown in \cite{hn}, every boundaryless compact manifold admits a pair of diffeomorphisms that generate a $C^1$ robustly minimal IFS. Ahead, we provide with a construction of such a pair of maps on $S^1$ with the additional properties of being mostly contracting and have a bounded $C^1$ distance to the identity. We consider for the rest of the article $S^1$ as the quotient of $[-1,1]$ under the identification $1\sim-1$.

\begin{lema}\label{SRM}
  Given $k>0$, there exists a family $\mathcal{F}=\{g_1,g_2\}$ in $Diff^1(S^1)$ such that $\max_{i \in \{1,2\}}\{ ||Id -g_i||\} <k$, $\max_{i \in \{1,2\}}\{ ||g'_i||\} <2$ and IFS$(\mathcal{F})$ is $C^1$ robustly minimal.
\end{lema}

\begin{proof}.\\
  Let $a \in \left(0, \frac{2}{3} \right)$ and $g_a:[-1,1] \to [-1,1]$ a real function given by $$g_a(x)=\left\{ \begin{array}{c}
                                                                     \left( \frac{2-3a}{2-2a}\right)(x+1)-1, \mbox{\ \ \ \ \  if $x \in \left[-1, -a \right]$}\\
                                                                    \frac 3 2 x, \mbox{ \ \ \ \ \ \ \ \ \ \ \ \ \ \ \ \ \ \ \ \ \ \ \ \ \  if $x \in \left[ -a, a \right]$} \\
                                                                    \left( \frac{2-3a}{2-2a}\right)(x-1)+1, \mbox{ \ \ \ \ \ \ \ \ \  if $x \in \left[ a, 1 \right]$}
                                                                  \end{array}\right.$$
  Then $g_a$ is a continuous piece-wise linear function that descends to $S^1$ as shown in Figure \ref{IFSmap}.
  Fix $a_0 \in \left(0, \frac {1}{ 26} \right)$ so that $||x-g_{a_0}(x)||< \frac{k}{2}$ and define $g:=g_{a_0}$.
  Let $g_1:S^1 \to S^1$ be a smooth approximation of $g$ such that $||g_1||\leq \frac 3 2$ and $g_1$ is a contraction on the complement of $\left( -2a_0,  2a_0 \right)$. Afterwards let $g_2:S^1 \to S^1$ be such that $g_2(x)=g_1(x- \frac{2}{13})$. \\ We claim $\mathcal{F}=\{ g_1, g_2\}$ is a family satisfying the announced properties.\\ To show minimality it is only needed to see that given any point $p$ in $S^1$, the orbit $\langle \mathcal{F}\rangle^+(p)$ is dense in $S^1$. Define $A:=\{ x \in S^1/\quad g_1 \mbox{ is a contraction at $x$}\}$ and $B:=\{ x \in S^1/ \quad g_2 \mbox{ is a contraction at $x$}\}$. We have $A \cup B=S^1$. Let $W$ be an open set in $S^1$, then either $g_1^{-1}(W)$ or $g_2^{-1}(W)$ is larger than $W$, so $\langle \mathcal{F}\rangle^{-1}(W)$ is strictly larger than $W$. Keep taking preimages until finding $n$ such that $\langle \mathcal{F}\rangle^{-n}(W)=S^1$ so $\langle \mathcal{F}\rangle^n(p) \in W$. To show robustness, observe that both $g_1$ and $g_2$ are Morse-Smale diffeomorphisms. Since these are structurally stable the proof above is robust. \\
  The last two properties claimed at the thesis of the Lemma are straightforward from the construction.
\end{proof}

\begin{figure}[ht]
\begin{center}
\includegraphics[scale=0.275]{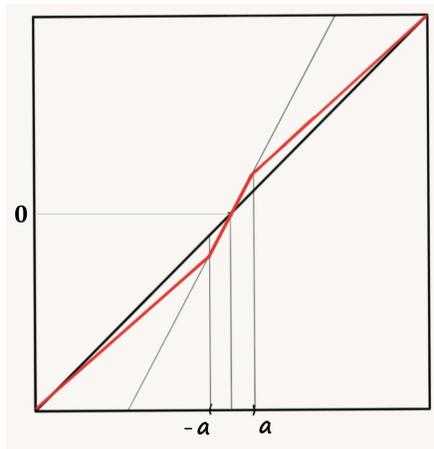}
\caption{$\tilde{g}_a:S^1 \to S^1$ is almost a contraction on $S^1$. }\label{IFSmap}
\end{center}
\end{figure}

Having stated all the preliminary facts needed to construct the example map satisfying the claim of Theorem \ref{main}, we proceed to it now in two steps. In Section 3 we define a $\T^n$ endomorphism (which we name $f$) that is $C^1$ robustly transitive. To achieve this goal, we use the result given by Lemma \ref{SRM} to create a blending region for $f$ supported on a strict subset $X$ of $\T^n$.\\
Once this construction is finished, we move on into Section 4 where the second step of the construction takes place by artificially introducing critical points inside the complement of $X$ in $\T^n$. The surgery is done in such a way that the critical point existence is robust and the blending region is unaffected, resulting in a new map (which we call $F$) that satisfies the claim at the thesis of Theorem \ref{main}.

\section{A regular endomorphism $f$ of $\T^n$.}

\subsection{Construction of $f$.}

 Consider the $n$ dimensional torus $\T^n=\R^n / [-1,1]^n$ and endow it with the standard riemannian (euclidean) metric. Let $\widehat{A} \in \mathcal{M}_n(\Z)$ be the diagonal matrix suggested below, with a unit in the last entry and all of the other elements being equal to $14$,

 \begin{equation}\label{matrizAgorro}
   \widehat{A}= \left(\begin{array}{ccccc}
 14 & 0 & 0 & \cdots & 0\\
0 & 14 & 0 & \cdots & 0 \\
\vdots & \vdots & \vdots & \ddots & \vdots \\
\vdots & \vdots & 0 & 14 & 0 \\
0 & \cdots & 0 &0 &1\\
\end{array}%
\right).
 \end{equation}

The matrix $\widehat{A}$ induces a regular endomorphism $A$ on the $n$-torus defined by
\begin{equation}\label{endoA}
  A:\T^n\rightarrow\T^n / \quad A(x_1,...,x_n)= (14x_1,14x_2,...,14x_{n-1},x_n).
\end{equation}

\begin{rk}.
  \begin{enumerate}
    \item  The construction could be carried on with any $\lambda \in \Z$ such that $|\lambda|>>1$; the choice of $14$ is made in the sake of simplicity and for a better understanding of the contents to follow.
    \item The construction can in fact be carried on with any $\lambda \in \Z$ such that $|\lambda|>1$ since there would be a power of $\widehat{A}$ such that the first entry would be larger than 14. It follows that the construction holds for any linear map in the isotopy class of maps with an eigenvalue of modulus larger than one.
    \item Observe that $A$ is a map $modulo$ $2$ even when we do not state it explicitly. The same convention applies for all maps of $\T^n$ defined along this work.
  \end{enumerate}
\end{rk}

For the rest of the construction, consider a decomposition of the $n$-torus given by $\T^n=\T^{n-1} \times S^1$; the map $A$ becomes $A:\T^{n-1} \times S^1 \rightarrow \T^n  / A(x,y)=(14x,y)$.\\

Define now the following cubes, subsets of the first factor $\T^{n-1}$: $K_0=\left[\frac{-1}{28},\frac{1}{28}\right]^{n-1}$, $K_1=\left[\frac{3}{28},\frac{5}{28}\right]^{n-1}$ and $K= \left(K_0 \cup K_1 \right)$. Take $\varepsilon=\frac{1}{1400}$ and define again cube sets $K_{0}^{\varepsilon}=\left[\frac{-1}{28}-\varepsilon,\frac{1}{28}+\varepsilon \right]^{n-1}$, $K_{1}^{\varepsilon}=\left[\frac{3}{28}-\varepsilon,\frac{5}{28}+\varepsilon \right]^{n-1}$ and $K^{\varepsilon}= \left(K_{0}^{\varepsilon} \cup K_{1}^{\varepsilon}\right) $.\\
Define next a smooth bump $\tilde{u}: \R \rightarrow \R$ given by Figure \ref{graficou} and let $u: \T^{n-1} \to \R$ be such that $u(x)=\frac{1}{n-1}\sum_{j=1}^{n-1} \tilde{u}(x_j)$. Observe that $u$ is smooth and satisfies $u_{|K}=1$ and $u_{|(K^\varepsilon)^c}=0$. Furthermore, $||\tilde{u}'||:=\max \{ |\tilde{u}'(x)|, x \in \R \} $ exists and $||\nabla u|| \leq ||\tilde{u}'||$.

\begin{figure}[ht]
\begin{center}
\includegraphics[scale=0.85]{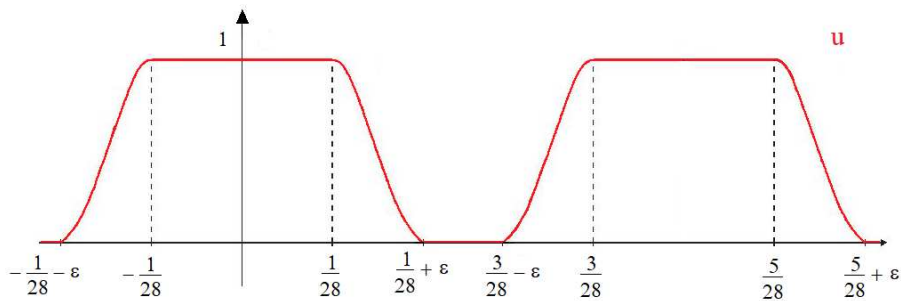}
\caption{Graph of $\tilde{u}$.}\label{graficou}
\end{center}
\end{figure}

Finally, fix a real number $0<\kappa <3$ and let $\mathcal{F}=\{g_1,g_2\}$ be the family given by Lemma \ref{SRM} for the second factor $S^1$, satisfying the properties claimed in the thesis of the theorem for $k= \frac{\kappa}{||\tilde{u}'||}$.\\

Define \begin{equation}\label{mapaefegorro}
         \hat{f}:K^\varepsilon \times S^1 \rightarrow \T^n  /  \hat{f}(x,y)=\left\{\begin{array}{c}
                                                            (14x , g_1(y)) \mbox{  if $x \in K_0^\varepsilon $} \\
                                                            (14x , g_2(y)) \mbox{  if $x \in K_1^\varepsilon $}
                                                          \end{array}\right.
       \end{equation}
and extend $\hat{f}$ to  \begin{equation}\label{mainmap}
                                f: \T^{n-1}  \times S^1 \rightarrow \T^n  /  f(x,y)=\left\{
                                \begin{array}{c}
                                  u(x).\hat{f}(x,y)+(1-u(x)).A(x,y) \mbox{ if $x \in K^{\varepsilon}$} \\
                                  A(x,y) \mbox{ if $x \notin K^{\varepsilon}$}
                                \end{array}\right.
                              \end{equation}

\begin{rk}\label{rkdinf}

 The following properties are straightforward to check:
\begin{enumerate}
  \item Calling $\hat{f}(x,y)=(14x,\hat{f}_2(y))$, then $f(x,y)=(14x,u(x).\hat{f}_2(y)+(1-u(x)).y)$.
   \item Since $\|Dg_1\| =\|Dg_2\| = \frac{3}{2}$ then $\|D\hat{f}_2\|<2$.
  \item By construction of $f$, $\| Id_{S^1} -\hat{f}_2\| \leq \frac{\kappa}{||u'||}$.
  \item The restriction $f_{|\left(K \times S^1\right)}=\hat{f}$.
  \item The restriction $f_{|{\left(K^\varepsilon \times S^1\right)}^c}=A$.
   \item $K \times S^1$ is a proto-blender for $f$ relative to $\left[ -\frac 1 2, \frac 1 2 \right]^{n-1} \times S^1$.

\end{enumerate}
\end{rk}

\subsection{Dynamics of $f$.}

 The most evident dynamical feature $f$ has is a strongly dominant expansion along the first factor $\mathbb{T}^{n-1}$. It follows that there exists a family of unstable cones for $f$ in the perpendicular direction to the last canonical vector $\vec{e_n}$, whereas $\vec{e_n}$ itself can be regarded as a central direction. We make a pause here to check the existence of such an unstable cone field for $f$.\\
Recall that for $x \in M$, we call  \textbf{cone} of parameter $a$, index $n-k$ and vertex $x$  to  $$C^{u}_{a} (x)=\left\{ (v_{1},...,v_{n}) \in T_x M / \frac{\Vert (v_{k+1},...,v_{n}) \Vert}{\Vert (v_{1}, v_2,...,v_k) \Vert}<a \right\}$$ and that $f$ admits an \textbf{unstable cone} of parameter $a$ and vertex $x \in M$ if there exists $C^u_a(x) \subset T_xM$ such that $\overline{D_x f(C^u_a(x)) }\setminus \{0\} \subset C^u_a(f(x))$.

\begin{lema}\label{conosf}
  The map $f$ defined by Equation (\ref{mainmap}) admits an unstable cone of parameter $\kappa$, index $1$ and vertex $(x,y)$ at every $(x,y) \in \T^n$.
\end{lema}
\begin{proof}
The differential of $f$ at $(x,y)$ is given by $$
                                                D_{(x,y)}f=\left( \begin{array}{cc}
                                                   14 & 0 \\
                                                   \nabla u(x).(\hat{f}_2(y)-y) & u(x).D_y\hat{f}_2+(1-u(x)).y
                                                 \end{array}\right).
                                               $$ Then for all vectors $(v_1,v_2)$ of the tangent space of $\T^{n-1} \times S^1$ at $(x,y)$ it is  $$D_{(x,y)}f(v_1,v_2)=\left( \begin{array}{c}
                                                   14v_1 \\
                                                   \left[\nabla u(x).(\hat{f}_2(y)-y)\right]v_1 + \left[u(x).D_y\hat{f}_2+(1-u(x)).y\right]v_2
                                                 \end{array}\right).$$

Consider all vectors $(v_1,v_2)$ in $C^u_\kappa(x,y)$ and let $(w_1,w_2):=Df_{(x,y)}(v_1,v_2)$, we see that the cone is unstable by computing $$\frac{||w_2||}{||w_1||}=\frac{\left|\left|\left[\nabla u(x).(\hat{f}_2(y)-y)\right]v_1 + \left[u(x).D_y\hat{f}_2+(1-u(x)).y\right]v_2 \right|\right|}{14|v_1|}\leq $$ $$ \leq \frac{||\nabla u||.||\hat{f}_2(y)-y||}{14}+\frac{(|u(x)|.||D_y\hat{f}_2||+|1-u(x)|.||Id||).\kappa}{14} \leq \frac{\kappa}{14}+\frac{4\kappa}{14}<\kappa,  $$ \\ where in the first inequality we apply triangular and that $ (v_1,v_2) \in C^{u}_{\kappa} (x,y)$ and in the second inequality we use:
\begin{enumerate}
\item $||\nabla u||.||\hat{f}_2(y)-y|| \leq \kappa$ by Remark \ref{rkdinf},
\item $ ||D\hat{f}_2||<2$ by Remark \ref{rkdinf},
\item $ \| Id\| \leq 2 $,
\item $ \max_{x \in \R} \{|u(x)|,|1-u(x)|\}\leq 1$.

\end{enumerate}

\end{proof}

\begin{lema}\label{estiraf}
  For all $ v \in C^{u}_{\kappa} (x,y)$ holds that $ \Vert D_{x}f (v) \Vert > 4 \Vert v \Vert $.
\end{lema}
\begin{proof}
Let $v=(v_{1},v_{2}) \in C^{u}_{\kappa} (x,y)$ and recall $0< \kappa <3$, then $$\left( \frac{\Vert D_{(x,y)}f (v_1,v_2) \Vert}{4.\Vert (v_1,v_2) \Vert}\right) ^{2} \geq \frac{(14.|v_{1}|)^{2}}{16.(|v_1|^2+||v_2||^2)}  \geq  \frac{196}{16\left(1+{\left(\frac{||v_2||}{|v_1|}\right)}^{2}\right)}> \frac{196 }{160} > 1. $$
\end{proof}

\begin{rk}\label{kappachico} Recall that if $B_k{(x,r)}$ denotes a ball of dimension $k$, the $k$-th. dimensional \textit{inradius} of a set $X$ is $ir_k(X):=\max_{x \in X}\{r>0/ \quad B_k{(x,r)} \subset X \}$.\\
  Since the definition of unstable cone is independent of the construction of $f$, $\kappa$ can be chosen small enough such that for all disks $\gamma$ satisfying $\gamma' \in C^u_{\kappa}(\gamma)$ at all times, then inradius and diameter of $\gamma$ can be identified. For the rest of the article assume $\kappa$ is small enough so that this property to hold.
\end{rk}

\begin{clly}\label{diametro}
  For all disks $\gamma$ such that for all $t$ where $\gamma$ is defined it holds that $\gamma'(t) \in C^u_{\kappa}(\gamma(t))$, the inradius satisfies $ir_k(f(\gamma))\geq 4ir_k(\gamma)$ for all $k \leq n-1$. 
\end{clly}

\begin{clly}\label{conorobusto}
  There exists a $C^1$ neighborhood $\mathcal{U}_f$ of $f$ such that all $g$ in $\mathcal{U}_f$ admit an unstable cone of parameter $\kappa$, index $1$ and vertex $(x,y)$ at every $(x,y) \in \T^n$ for which Corollary \ref{diametro} holds.
\end{clly}

We highlight now some of the other relevant geometrical and dynamical features the map $f$ possesses. All of them are straightforward to check:

\begin{rk}\label{dynF}
  1) $K^\varepsilon  \subset [-\frac{1}{2}-\varepsilon,\frac{1}{2}+\varepsilon]^{n-1} $.\\
  2) $f(K_0^\varepsilon \times S^1) \cap f(K_1^\varepsilon \times S^1) \supset [-\frac{1}{2}-\varepsilon,\frac{1}{2}+\varepsilon]^{n-1} \times S^1$.\\
  3) The set $K^\varepsilon \times S^1$ is a protoblender for $f$ relative to $[\frac{-1}{2}-\varepsilon,\frac{1}{2}+\varepsilon]^{n-1} \times S^1$.\\
  4) The points $(0,...,0,1)$ and $(\frac{2}{13},...,\frac{2}{13},\frac{15}{13}) $ are saddle fixed points for $f$, and the points $(0,...,0,0) $ and $(\frac{2}{13},...,\frac{2}{13},\frac{2}{13})$ are repelling fixed points for $f$. \\
  5) The local unstable manifold at $(0,1)$ is $W^u_{loc}(0,0)=(-\varepsilon,\varepsilon)^{n-1}\times\{1\}$.\\
  6) If $B \subset (1-a_0,1+a_0) \subset S^1$ satisfies $B:=W^s_{loc}(1)$ for $g_1$, then the local stable manifold for $f$ at $(0,1)$ is $W^s_{loc}(0,1)=\{(0,...,0)\}\times B$.
\end{rk}

We prove now that both the local stable and unstable manifolds at $(0,1)$ are dense in $\T^n$. This will yield $f$ is $C^1$ transitive.\\ In the sake of simplicity, from now on the points $(0,..,0)$ and $(\frac{2}{13},...,\frac{2}{13})$ in $\mathbb{T}^{n-1}$ will be referred to as $0$ and $\frac{2}{13}$ when there is no risk of confusion.

\begin{lema}\label{umd}
  The unstable set $W^u(0,1)$ is forward $f$-dense in $\T^n$.
\end{lema}
\begin{proof}
Let $V=V_1 \times V_2$ be any open set in $\T^n=\T^{n-1} \times S^1$. We show that there exists a point in $W^u_{loc}(0,1)=(-\varepsilon,\varepsilon)^{n-1}\times\{1\}$ with a forward iterate in $V$.\\
Let $f$ be $f(x,y)=(14x,f_2(y))$. Since $f_1(x)=14x$ expands, there exists a natural number $k$ such that $f^{k} (W^u_{loc}(0,1))\supset \T^{n-1} \times \{f_2^k(1)\}$. Since IFS$(\mathcal{F})$ is minimal, there exists a point in the orbit of $\langle \mathcal{F} \rangle ^{+}\left(f_2^{k}(1)\right)$ that intersects $V_2$ at, let's say, $f_2^{k+j}(1)$. This implies $f^{k+j}(W^u_{loc}(0,1)) \cap \T^{n-1} \times V_2 \neq \emptyset$ so $f^{k+j}(W^u_{loc}(0,1)) \cap  V \neq \emptyset$.
\end{proof}

\begin{lema}\label{smd}
  The stable set $W^s(0,a_1)$ is backwards $f$-dense in $\T^n$.
\end{lema}
\begin{proof}
Let $V=V_1 \times V_2$ be any open set in $\T^{n-1} \times S^1$ and let $W^s_{loc}(0,1)=\{0\}\times B$ where $B=W_{loc}^s(1)$ for $g_1$. Let $p=(p_1,...,p_{n-1},p_n) \in V$ and a well defined disk $\gamma:(-r,r)^{n-1} \rightarrow V /\quad \gamma(t_1,...,t_{n-1})=(p_1+t_1,...,p_{n-1}+t_{n-1}, p_n)$. Since for all $t \in (-r,r)^{n-1}$ and all $v \in T_{\gamma(t)}V$ the differential $D_t\gamma(v_1,...,v_n)=(v_1,...,v_{n-1},0)$, it holds that $\gamma'(t) \in C^u_{\kappa}(\gamma(t))$ at all times. By Corollary \ref{diametro}, for all $k \in \N$, $ir_{n-1}(f^{k}(\gamma))\geq 4^{k}ir_{n-1}(\gamma)$. Since $ir_{n-1}(\mathbb{T}^{n-1})\leq \sqrt{n}$, there exists $k_0 \in \mathbb{N}$ and $z \in S^1$ such that $f^{k_0}(\gamma) \supset  \T^{n-1} \times \{z\} $. Again, since IFS$(\mathcal{F})$ is minimal there exists a branch of the orbit $\langle \mathcal{F} \rangle^+(z)$ that enters $B$, say, at $\langle \mathcal{F} \rangle^j(z)$. In turn, $f^{k_0+j}(\gamma) \cap \{0\} \times B \neq \emptyset$. Therefore, it exists a point in $V$ (in $\gamma$) with a forward iterate entering in $W^s_{loc}(0,1)$.
\end{proof}

\begin{thm}\label{fisT}
  The map $f$ defined by Equation (\ref{mainmap}) is robustly transitive.
\end{thm}
\begin{proof}
   According to Lemmas \ref{umd} and \ref{smd} the open set $U \times B$ is an open neighborhood of $(0,1)$ that is dense under forward and backward iteration by $f$. Hence, by an inclination argument, it holds that for all open sets $A$ and $B$ in $\T^n$ there exists a natural number $k$ such that $f^k(A) \cap B \neq \emptyset$ which yields transitivity for $f$.
\end{proof}

 We turn our attention now to prove robustness of Theorem \ref{fisT}. We begin with a series of considerations about the perturbation that are required to understand its dynamics.\\ Start with a small $C^1$ neighborhood $\mathcal{V}_{f}$ of $f$. Notice first that after item (2) at Remark \ref{dynF}, all maps in $\mathcal{V}_{f}$ preserve the blending region $K \times S^1$. As well, after Corollary \ref{conorobusto}, all maps in $\mathcal{V}_{f}$ admit a field of unstable cones for which Corollary \ref{diametro} holds.\\

Recall that $\{0\} \times S^1$ and $\{\frac{2}{13}\} \times S^1$ are $f$-invariant disjoint submanifolds to which the restriction of $f$ configures a minimal iterated function system. \\
Let $g$ be an arbitrary map in $\mathcal{V}_{f}$, by item (4) at Remark \ref{dynF}, $g$ admits two saddle fixed points we name as $0'$ and $\frac{2}{13}'$ which are the continuation points of the saddles of $f$. Recall that $f_{
|\{0\} \times S^1}=g_1$ is a Morse-Smale diffeomorphism, so by their stability there exists $\varepsilon >0$ such that for every $g: \{0\} \times S^1 \to \{0\} \times S^1$ satisfying $||g-g_1||<\varepsilon$, there exists a homeomorphism $H$ such that $H \circ g_1= g \circ H$.\\

Notice that $f([\frac{-1}{2.14^2}, \frac{1}{2.14^2}]^{n-1} \times S^1)= K_0 \times S^1$. It yields that the preimage map  $f^{-1}:K_0 \times S^1 \to [\frac{-1}{2.14^2}, \frac{1}{2.14^2}] \times S^1$ is a diffeomorphism satisfying that $\bigcap_{n \in \mathbb{N}} f^{-n}(K_0 \times S^1)=\{0\}\times S^1$, which yields $f^{-1}$ is normally hyperbolic at $\{0\}\times S^1$. Apply Theorem \ref{NHSP} to see that $\bigcap_{n \in \mathbb{N}} g^{-n}(K_0 \times S^1)= N$ is a unique $g$-invariant submanifold of $\T^n$ and that there exists $h_1: \{0\}\times S^1 \to N$ a $C^1$ diffeomorphism that can be set as close to the canonical inclusion as desired by choosing $g$ close enough to $f$. Take then $\mathcal{V}_f$ and reduce it as needed as to satisfy $||g_1 - h_1^{-1}\circ g_{|N} \circ h_1||< \varepsilon$. In turn there exists $H_1: \{0\}\times S^1 \to N$ a homeomorphism such that $H_1 \circ g_1 =g_{|N} \circ H_1$ after the stability argument about $g_1$ right above. Observe that $H_1= h_1 \circ H$ and that $0' \in N$. An identical argument yields that there exists at $\frac{2}{13}'$ a $g$-invariant submanifold $ N'$ which is $C^1$-diffeomorphic to $\{\frac{2}{13}\}\times S^1$ and a conjuating homeomorphism $H_2: \{\frac{2}{13}\}\times S^1 \to  N'$ such that $H_2 \circ g_2 =g_{|N'} \circ H_2$.\\ We proved that $g_{|N}$ is conjugate to $g_1$ and that $g_{| N'}$ is conjugate to $g_2$.

\begin{rk}\label{doesnotmisspoints}
  Since $N=\bigcap_{n \in \mathbb{N}} g^{-n}(K_0 \times S^1)$, for every continuous $(n-1)$-disk $\gamma$ that crosses $K_0$ intersecting it in every side of its boundary, then $\gamma \cap N \neq \emptyset$. Moreover, since $N$ is $C^1$ close to $S^1$ it can be regarded as a 'vertical' submanifold. 
\end{rk}

\begin{rk}\label{cuencasencoordenadas}
  The maps $g_{|N}$ and $g_{| N'}$ do not configure an IFS since they are not defined on the same support, but it comes from Lemma \ref{SRM} and the conjugations stated above that if $\pi:\mathbb{T}^{n-1} \times S^1 \to S^1/\quad \pi(x,y)=y$ is the canonical projection then $\pi\left(\{ g_{| N} \mbox{ is a contraction or } g_{|N'} \mbox{ is a contraction}\}\right)=S^1$.
\end{rk}

We proceed now to prove a last series of lemmas that will lead to the proof of robust transitivity for the map defined by Equation (\ref{mainmap}). We start showing that Remark \ref{cuencasencoordenadas} yields that any neighborhood of $0'$ has a preimage by a power of the perturbation $g$ that projects surjectively onto the second factor $S^1$.

\begin{lema}\label{gestirapatra}
   For every $U$ open neighborhood of $0'$ in $\mathbb{T}^n$ and $U' \subset U$ any open subset, there exists $m_0 \in \mathbb{N}$ such that $\pi \left( g^{-m}(U')\right)=S^1$ for all $m \geq m_0$.
\end{lema}

\begin{proof}
  Begin noticing that if $\pi: \mathbb{T}^{n-1} \times S^1 \to S^1$ is the canonical projection onto the second factor, $||(\pi \circ g) - (\pi \circ f)||<||g-f||$ holds. By Remark \ref{rkdinf} and Definition \ref{mapaefegorro} we conclude that $(\pi \circ f)_{|K_0 \times S^1}={g}_1$ and that $(\pi \circ f)_{|K_1 \times S^1}={g}_2$. Thereafter, since ${g}_1$ and ${g}_2$ are Morse-Smale diffeomorphisms, $\mathcal{V}_f$ can be reduced until $(\pi \circ g)_{|K_0 \times S^1}$ is conjugate to ${g}_1$ and $(\pi \circ g)_{|K_1 \times S^1}$ is conjugate to ${g}_2$. The same accounts for the inverse image map $g^{-1}$.\\
   Let then $U$ be any open neighborhood of $0'$ and $U' \subset U$ any open subset. Recall that $g$ preserves the blending region $K \times S^1$, so the set $g^{-1}(U')$ contains a preimage component in $K_0 \times S^1$ and another one in $K_1 \times S^1$. By the conjugations above and the proof of Lemma \ref{SRM}, $\pi(g^{-1}(U'))$ is a curve with strictly larger length than $\pi(U')$ in either one of those components. Call $X_1$ the component that has strictly larger projection and consider $g^{-1}(X_1) \subset g^{-2}(U')$. Find again a component of $g^{-1}(X_1)$ in $K_0 \times S^1$ and another one in $K_1 \times S^1$, project them to $S^1$ and choose $X_2$ the one that projects with strictly larger length. Repeat the process until finding, via the same argument at Lemma \ref{SRM}, a natural number $m_0$ such that $\pi(g^{-m}(U'))=S^1$ for all $m \geq m_0$.
\end{proof}

We are in condition to state and prove the last two lemmas required for the proof of robust transitivity. For the local stable and unstable manifolds of $g$ at $0'$, let $\rho$ be a transversal curve to $\mathbb{T}^{n-1}$ centered at $0'$ contained in the local stable manifold of $g_{|N}$ in $ N$. It holds that $W^s_{loc}(0'):=\{\rho(t)\}_{t \in (-\varepsilon,\varepsilon)} \subset N$ with $\rho(0)=0'$.\\ In the same fashion there exists a $(n-1)$-disk $\lambda$ which is transversal to $N$ such that $W^u_{loc}(0'):=\{ \lambda(t)\}_{t \in (-\varepsilon,\varepsilon)^{n-1}}$ with $\lambda(0)=0'$.

\begin{lema}\label{umdg}
  $W^u_{loc}(0')$ is forward $g$-dense in $\mathbb{T}^n$.
\end{lema}
\begin{proof}
  Let $U=U_1 \times U_2$ be a small open neighborhood of $0'$ in $\mathbb{T}^{n}$ and suppose $W^u_{loc}(0')$ is not forward $g$-dense in $U$. It means there exists an open subset $U' \subset U$ such that for all $n \in \mathbb{N}$, $g^n\left(\lambda \right) \cap U'=\emptyset$. Apply Lemma \ref{gestirapatra} to find $m_0$ such that $\pi(g^{-m}(U'))=S^1$ for all $m \geq m_0$ and consider the preimage $g^{-m_0-1}(U')$ contained in $K_0 \times S^1$. Since $0'$ is a saddle for $g$ which only contracts in $g_{|N}$, there exists $k_0 \geq m_0+1$ such that $g^{-k_0}(U')\cap \lambda \neq \emptyset$. This gives a contradiction. Consequently, $\lambda$ is forward $g$-dense in $U$. By expansion in the first factor we have that $\lambda$ is forward $g$-dense in $\mathbb{T}^{n-1} \times U_2$. Since $\frac{2}{13}'$ is a saddle for $g$ which contracts only in $g_{|N'}$ and Remark \ref{doesnotmisspoints} gives that by this point for some $m \in \mathbb{N}$, $g^m(\lambda) \cap N' \neq \emptyset$, then $\lambda$ is forward $g$-dense in $\mathbb{T}^{n-1} \times \bigcup_{k \in \mathbb{N}} (g_{|N'})^k(U_2)$ and by an analogous argument it is forward $g$-dense in $\mathbb{T}^{n-1} \times \bigcup_{j \in \mathbb{N}} (g_{|N})^j ( \bigcup_{k \in \mathbb{N}} (g_{| N'})^k(U_2))$ which is the whole $\mathbb{T}^n$.
\end{proof}

The following lemma relies on the following fact from Euclidean spaces: if $\vec{v}$ is a vertical vector and $\mathcal{H}$ is a horizontal hyperplane, they intersect in a robust way. 

\begin{lema}\label{smdg}
  $W^s_{loc}(0')$ is backwards $g$-dense in $\mathbb{T}^n$.
\end{lema}

\begin{proof}
  Let $U=U_1 \times U_2$ be any open set in $\mathbb{T}^{n-1} \times S^1$ and $\gamma$ a $(n-1)$-disk in $U$ such that $\gamma'$ belongs to $C^u_{\kappa}(\gamma)$ the cone field of $g$ at $\gamma$ at all times. Therefore by Corollary \ref{diametro}, for all $m \in \mathbb{N}$, $ir_{n-1}(g^{m}(\gamma))\geq 2^mir_{n-1}(\gamma)$. Since $ir_{n-1}(\mathbb{T}^{n-1}) \leq \sqrt{n}$ and $C^u_{\kappa}$ is 'horizontal', Remark \ref{doesnotmisspoints} ensures that both $g^{m_0}(\gamma) \cap N \neq \emptyset$ and $g^{m_0}(\gamma) \cap N' \neq \emptyset$.\\ To finish, if $0' \notin g^{m_0}(\gamma)$, since $0'$ is attracting for $g_{|N}$ then $g^{m_0+j}(\gamma) \cap W^s_{loc}(0') \neq \emptyset$ for some $j \in \mathbb{N}$. If, on the contrary, $0' \in g^{m_0}(\gamma)$ then $g^{m_0}(\gamma) \cap N' \neq \{ \frac{2}{13}' \}$ since $\kappa$ is small and there is a large distance between $0'$ and $\frac {2} {13}'$ (due to the geometry of $\mathcal{F}$ at Lemma \ref{SRM} together with the conjugations explained right after Theorem \ref{fisT}). In the latter case, since $\frac{2}{13}'$ is a sink for $g_{|N'}$ it holds that  $g^{m_0+j}(\gamma) \cap W^s_{loc}(\frac{2}{13}') \neq \emptyset$ and consequently by attraction of $0'$ for $g_{|N}$ then $g^{n_0+j+k}(\gamma) \cap \rho \neq \emptyset$. In either case, there exists $m \in \mathbb{N}$ such that  $g^{m}(\gamma) \cap \rho \neq \emptyset$ which means that $g^{-m}(\rho) \cap \gamma \neq \emptyset$ which gives the claim at the thesis since $\gamma \subset U$.
\end{proof}

\begin{thm}\label{fisRT}
  The map $f$ defined by Equation (\ref{mainmap}) is robustly transitive.
\end{thm}
\begin{proof}
  Let $g \in \mathcal{V}_f$. Since $g$ satisfies Lemma \ref{umdg} and Lemma \ref{smdg}, then $g$ satisfies Theorem \ref{fisT}. Since $g$ is an arbitrary map in $\mathcal{V}_f$, then $f$ is $C^1$ robustly transitive.
\end{proof}

\section{A singular endomorphism $F$ of $\T^n$.}
 Now that we have defined a robustly transitive endomorphism $f$ given by a blending region contained in $K \times S^1$, we procceed to the second step of the construction by (robustly) artificially introducing critical points in the complement of $K^{\varepsilon} \times S^1$. The technique used to introduce the critical points is inspired by the construction carried on in \cite[Section 2.2]{ilp}. Once the surgery over $f$ is performed, the map $F$ announced at Theorem \ref{main} arises. \\

 As a short set of preliminaries, we provide first the definitions of the singularities of any map $h:M \rightarrow M$; recall that $M$ denotes a real manifold of dimension $m$.
 We say that $x \in M$ is a \textbf{critical point} or \textbf{singularity} for $h$ if the differential map at $x$, $D_x h$ is not surjective. Observe that $x$ is a singularity for $h$ if and only if the rank of the jacobian matrix satisfies $rk(D_x h) < m$, if and only if the determinant $det(D_xh) = 0$. The \textbf{critical set} of $h$ is $ S_h = \{ x \in M / rk(D_x h) < m \} $. We say that $h$ is a \textbf{singular endomorphism} if the critical set $S_h$ is non empty; and we say that $h$ is a \textbf{persistently singular endomorphism} if there exists a neighborhood $\mathcal{U}_h$ of $h$ in the $C^1$ topology such that all $ g \in \mathcal{U}_h$ satisfy $S_g \neq \emptyset$.

\subsection{Construction of $F$}\label{constructionofF}

{\textit{Sketch of the construction:} We choose a point not in $K^\varepsilon \times S^1$ and set a ball centered at this point, inside the complement of $K^\varepsilon \times S^1$. By means of standard surgical procedures, we perturb $f$ to introduce a set of critical points inside the ball and with the additional property that the resulting critical set is persistent. Since the surgery does not affect the blending region $K \times S^1$, the robust transitivity of the map $f$ defined by Equation (\ref{mainmap}) is inherited by the new map. We call the new map $F$, and it satisfies the claim at the title of the article. \\

\begin{figure}[ht]
\begin{center}
\subfigure[]{\includegraphics[scale=1]{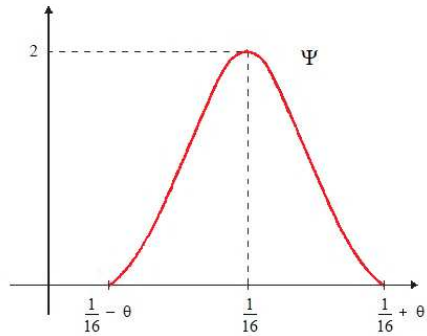}}
\subfigure[]{\includegraphics[scale=0.37]{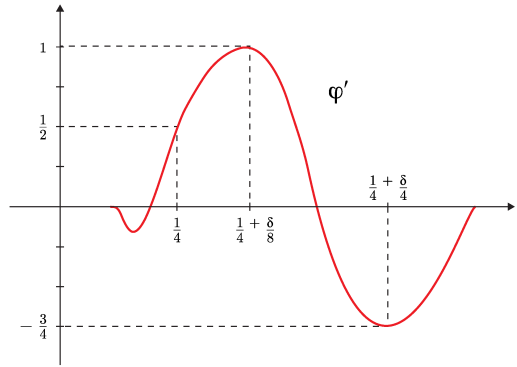}}
\caption{Graphs of $\psi$ and $\varphi'$ (taken from \cite{mo})}\label{figura11}
\end{center}
\end{figure}

Let $p=(\frac{1}{4},0,...,0,\frac{1}{4})\in\T^{n}$. Our goal is to define a ball of center $p$ to perform a perturbation in order to obtain the map $F$ we seek. To achieve this goal we need to fix a series of technical parameters; the choice to set all of them at the same time and at the beginning of the construction is in expectance of avoiding darkness and of that it will be clear how they depend on each other.\\

Start with $r>0$ satisfying that the ball $B_{(p,r)} \cap \left( K^\varepsilon \times S^1 \right)= \emptyset$, this is possible since $p \notin \left( K^\varepsilon \times S^1 \right)$. Fix a second parameter $\theta$ such that  $0< \theta  < \frac{r}{2}$  and define a smooth ($C^{\infty}$) function $\psi :\R\to\R$ with a unique critical point at $\frac{1}{16}$, with $\psi (\frac{1}{16})=2$ and $\psi (x)=0$ for all $ x$ in the complement of $(\frac{1}{16}-\theta ,\frac{1}{16}+\theta  )$;  and an axis of symmetry in the line $x=\frac{1}{16}$ as shown in Figure \ref{figura11} (a) .

Set finally a last parameter $\delta $, with  $0<\delta < 2\theta $ verifying the following condition: since the derivative of $\psi$ is bounded once $\theta$ has been fixed, name the bound as $ m_\psi := m_\psi(\theta) = max  \{ | \psi'(x) |, x \in \R \} $ and impose on $\delta$ that $2.m_\psi.r.\delta<11\kappa$.  \\

Having fixed $\delta$, consider another smooth function $\varphi:\R\to\R$ such that:

\begin{enumerate}
\item $\varphi' $ is as in Figure \ref{figura11} (b).
\item $\frac{-3}{4}\leq \varphi'(x)\leq 1$ for all $x \in \R$. This gives $| \varphi'(x)| \leq 1$ for all $x \in \R$.
\item $\varphi' (x)=0$ for all
$ x \notin [\frac{1}{4}-\frac{\delta}{4},\frac{1}{4}+\frac{3\delta}{4} ]$.
\item $\varphi' (\frac{1}{4})=\frac{1}{2}$, $\varphi' (\frac{1}{4} +\frac{\delta}{8}  )=1$, $\varphi' (\frac{1}{4} +\frac{\delta}{4}  )=-\frac{3}{4}$, $\varphi (\frac{1}{4})=0$.

\end{enumerate}
\begin{rk}\label{rkdelta}
  $max \{ |\varphi (x)|: \ x \in \R \}\leq \delta.$
\end{rk}

We are now in condition to define a perturbation of $f$ in the direction of the last canonical vector $\vec{e_n}$ that depends on $r$,$\theta$ and $\delta$ which by simplicity we call only $F$ and is defined at $x=(x_1,...,x_n)$ as \begin{equation}\label{mapaefe}
           F_{r,\theta, \delta}:\T^n\to \T^n / F(x)= \left\{\begin{array}{c}
                                                            f(x) \mbox{\ \ \ \ \ \ \ \ \ \ \ \ \ \ \ \ \ \ \ \ \ \ \ \ \ \ \ \ \ \ \ \ \ \ \ \ \   if $x \notin B_{(p,r)} $} \\
                                                            A(x) - \varphi(x_n).\psi\left( \sum_{j=1}^{n-1}x_j^2 \right).\vec{e_n} \mbox{\ \  if $x \in B_{(p,r)} $}
                                                          \end{array}\right. .
         \end{equation}

\begin{rk}\label{rkF}.

\begin{enumerate}
            \item For all $x \notin B_{\left(p,\frac{3\delta}{4}\right)}$ it holds that $F(x)=f(x)$.
            \item For all $x \notin K^{\varepsilon}\times \T^{n-1}$ it holds that $f(x)=A(x)$.
          \end{enumerate}

\end{rk}
To make the reading easier we will denote $\varphi(x_n)$ as $\varphi$ and $\psi \left( \sum_{j=1} ^{n-1} x_j^2 \right)$ as $\psi$ omitting the evaluations appearing on the definition.

\begin{lema}\label{Fps}

  The endomorphism $F$ defined by Equation (\ref{mapaefe}) is persistently singular.
\end{lema}
\begin{proof}
Start computing the differential $D_x F$ at $x=(x_1,...,x_n) \in B_{(p,r)} $ to get

\begin{equation}\label{ecuacion1}
D_xF =
\left( \begin{array}{ccccc}
14 & 0 & \cdots & 0 & 0 \\
0 & 14 & \cdots & 0 & 0 \\
\vdots & \vdots & & \vdots & \vdots \\
0 & 0 & \cdots & 14 & 0 \\
-2.x_1.\varphi.\psi' & -2.x_2.\varphi.\psi'& \cdots & -2.x_{n-1}.\varphi.\psi' & 1-\varphi'.\psi \\
\end{array}
\right).
\end{equation}\\

Since the critical set of $F$ is defined as $S_F=\{x \in \T^{n} /det ( D_xF )=0\}$, Equation (\ref{ecuacion1}) provides $det(D_xF)=14^{n-1} \cdot (1-\varphi'.\psi) $. In turn, $S_F=\{x \in \T^{n} / 1-\varphi'.\psi =0\rbrace.$\\
Notice that $S_F$ is not empty since $p=(\frac{1}{4},0,...,0,\frac{1}{4}) \in S_F$. To prove that $S_F$ is persistent, consider the points $q_1=(\frac{1}{4},0,...,0, \frac{1}{4} + \frac{\delta}{4})$ and $q_2=(\frac{1}{4},0,...,0, \frac{1}{4} + \frac{\delta}{8})$ both in $B_{(p,r)}$. Evaluate determinants \textit{det}$(D_{q_1}F)=\frac 5 2 . 14 ^{n-1}$ and \textit{det}$(D_{q_2}F)= -14 ^{n-1}$. Therefore, for a neighborhood $\mathcal{U}_F \in C^{1}$ of radi $1$, every $ g \in \mathcal{U}_F$ satisfies $S_g\neq\emptyset$.

\end{proof}

\subsection{Dynamics of $F$.}
We turn now to the last part of the article where we show that $F$ is $C^1$ robustly transitive. To prove it, observe first that after Remark \ref{rkF}, Lemma \ref{umdg}  holds for $F$ automatically. If we prove that Lemma \ref{smdg} also holds for $F$, then we can apply the same reasoning of Theorem \ref{fisRT} to $F$ to have the result. Notice that for Lemma \ref{smdg} to hold for $F$ we only need to show that $F$ admits an unstable cone $C_{\kappa}^u(x)$ of index $1$ at every point $x \in B_{(p,r)}$ which satisfies that for all $(n-1)$-disk $\gamma$ with $\gamma' \in C_{\kappa}^u(\gamma)$ then $ir_{n-1}(F(\gamma))\geq 4ir_{n-1}(\gamma)$.

\begin{lema}\label{conoF1}
  For all $x \in B_{(p,r)}$ it holds that $C^{u}_{\kappa} (x)$ is an unstable cone for $F$.
\end{lema}

We adopt the following notation for the proof, $\tilde{v}:=(v_{1},v_{2},...,v_{n-1})$ whenever $v=(v_{1},v_{2},...,v_{n})$.

\begin{proof}
From Equation (\ref{ecuacion1}) we have for all $ v=(v_{1},v_{2},...,v_{n}) \in C^{u}_{\kappa} (x):$
$$D_{x}F (v) = (14v_{1},14v_{2},...,14v_{n-1},-2.{\langle \tilde{x},\tilde{v}\rangle}.\varphi.\psi' + v_{n}.(1-\varphi'.\psi)). $$

Call $ u:=D_{x}F (v)=(u_1,..,u_n)$ and perform calculations, we have:  $$ \frac{\vert u_{n}\vert}{\Vert \tilde{u} \Vert}  = \frac{\vert -2.{\langle \tilde{x},\tilde{v}\rangle}.\varphi .\psi'+ v_{n}.(1-\varphi' .\psi) \vert}{\Vert  14\tilde{v} \Vert} \leq  \frac{2.\Vert \tilde{x} \Vert.\Vert \tilde{v} \Vert.\vert \varphi \vert.\vert \psi' \vert}{14\Vert \tilde{v}\Vert} + \frac{\vert 1-\varphi' .\psi \vert. \vert v_{n}\vert}{14\Vert \tilde{v} \Vert} < \frac{2.r.\delta.m_\psi}{14} +\frac{3\kappa}{14} < \kappa. $$ \\ Above, for the first inequality we use triangular and Cauchy-Schwarz; for the second one we use:
\begin{enumerate}
\item $ v \in C^{u}_{\kappa} (x)$,
\item $\Vert \tilde{x} \Vert \leq \Vert x \Vert < r$,
\item $ \vert \varphi \vert \leq \delta $ by Remark \ref{rkdelta},
\item $ \vert \psi' \vert < m_\psi $,
\item $ \vert 2-\varphi' .\psi \vert \leq 3 $ since $ \frac{-3}{4} \leq \varphi' \leq 1 $ and $ 0 \leq \psi \leq 2$.

\end{enumerate}
And for the third one we use the condition $2.m_\psi.r.\delta<11\kappa$ imposed over $\delta$.
\end{proof}
\begin{lema}\label{conoF2}
  For all $x \in B_{(p,r)}$ and all $ v \in C^{u}_{\kappa} (x)$ it holds that $ \Vert D_{x}F (v) \Vert > 4 \Vert v \Vert $.
\end{lema}
\begin{proof}
  Let $\tilde{v} \in \R^{n-1}$ and $v_n \in \R$ such that $ v=(\tilde{v},v_n) \in C^{u}_{\kappa} (x) \subset T_x \T^n$. Recall that $0<\kappa<3$ and compute: $$\left( \frac{\Vert D_{(x,y)}F (\tilde{v},v_n) \Vert}{4.\Vert (\tilde{v},v_n) \Vert}\right) ^{2} \geq \frac{\Vert 14\tilde{v}\Vert^{2}}{16.(\Vert \tilde{v}\Vert^{2}+|v_n|^2)}   \geq  \frac{196}{16.\left(1+\frac{|v_n|^2}{||\tilde{v}||^2}\right)} > \frac{196 }{160}>  1. $$ \end{proof}

\begin{lema}\label{Frt}
  The map $F$ defined by Equation (\ref{mapaefe}) is $C^1$ robustly transitive.
\end{lema}
\begin{proof} From Lemmas \ref{conoF1} and \ref{conoF2} we conclude that Lemma \ref{smdg} holds for $F$. It was already mentioned that Lemma \ref{umdg} holds for $F$. Consequently, Theorem \ref{fisRT} holds for $F$. \end{proof}

We are now in condition to give the proof to the main Theorem of the article:

\begin{proof}[ \textbf{Proof of Theorem \ref{main}}]
Define $\mathcal{U}_1 \in C^1$ an open neighborhood of $F$ where Lemma \ref{Fps} holds and $\mathcal{U}_2 \in C^1$ an open neighborhood of $F$ where Lemma \ref{Frt} holds. Then, all maps belonging to $\mathcal{U}_F=\mathcal{U}_1 \cap \mathcal{U}_2$ are $C^1$ robustly transitive and have nonempty critical set.
\end{proof}

\section{Final Remarks} The example exhibited in this article shows the existence of $C^1$ robustly transitive maps displaying critical points on any dimension. Yet, many open questions remain: \textit{Is $\T^n$ the only high dimensional manifold supporting such a map? Is it possible to extend this type of construction to other quotient manifolds? Would it be possible to carry on the proof starting from a matrix whose linear induced map belongs to a different isotopy class? Is it possible that a fiber bundle (in stead of a product) admits a construction of this type?} just to mention some of them.

\section*{Acknowledgements} The author would like to give thanks to Prof. Jorge Iglesias for fruitful conversations regarding this problem and to Prof. Roberto Markarian and Prof. Lorenzo J. D\'iaz for their generous attitude towards the author and his work. Finally, the author wants to thank Prof. Rafael Potrie for his reading and comments regarding the construction which helped improving the outcome of the article.

\end{document}